\documentclass[11pt,reqno]{amsart}
    \usepackage{amsfonts}
    \usepackage{amsmath}
    \usepackage{amssymb}
    \usepackage{amscd}
    \usepackage{amstext}
    \usepackage{amsthm}
    \newtheorem{thm}{Theorem}[section]
    \newtheorem{Lem}[thm]{Lemma}
    \newtheorem{Prop}[thm]{Proposition}
    \newtheorem{Def}[thm]{Definition}
    \newtheorem{Cor}[thm]{Corollary}

        \newcommand{\A}{\mathcal A}

    \newcommand{\tsr}{\mathrm{tsr}\,}
    \newcommand{\csr}{\mathrm{csr}\,}
    \newcommand{\gsr}{\mathrm{gsr}\,}
    \newcommand{\Lg}{\mathrm{Lg}\,}
    \newcommand{\diag}{\mathrm{diag}\,}
    \newcommand{\rank}{\mathrm{rank}\,}
    
    \newcommand{\HO}{\mathrm{H}}
    \def\Ker{\mathrm{Ker}\,}
    \def\ran{\mathrm{Ran}\,}
    \def\det{\mathrm{det}\,}
    \def\tr{\mathrm{tr}\,}
    
    \def\dcup{\mathop{\displaystyle\bigcup}\limits}

\begin{document}
\title{Approximate diagonalization of self--adjoint matrices over $C(M)$}
\author{Yifeng Xue}
\address{ Department of Mathematics, East China Normal University\newline
\hspace*{3mm} Shanghai 200241, P.R. China}
\email{yfxue@math.ecnu.edu.cn}

\begin{abstract}
Let $M$ be a compact Hausdorff space. We prove that in this paper, every self--adjoint matrix
over $C(M)$ is approximately diagonalizable iff $\dim M\le 2$ and $\HO^2(M,\mathbb Z)\cong 0$.
Using this result, we show that every unitary matrix over $C(M)$ is approximately diagonalizable
iff $\dim M\le 2$, $\HO^1(M,\mathbb Z)\cong\HO^2(M,\mathbb Z)\cong 0$ when
$M$ is a compact metric space.
\end{abstract}
\thanks{Project supported by Natural Science Foundation of China (no.10771069) and Shanghai
Leading Academic Discipline Project(no.B407)}
\maketitle

\baselineskip 17pt
\section{Introduction}

For a unital $C^*$--algebra, we write $U (\A )$ (resp. $U_0(\A)$)
to denote the unitary group of $\A$ (resp. the connected component of the unit
in $U (\A )$). We also denote by ${\text M}_n(\A)$ the matrix algebra of $n\times n$ over $\A$.
Set $U_1(\A)=U(\A)$ (resp. $U_1^0(\A)=U_0(\A)$) and
$
U _n(\A )=U ({\text M}_n(\A )),\,U ^0_n(\A)=U _0({\text M}_n(\A )).
$

In \cite{K}, Kadison proved that for a Von Neunnam algebra $\A$, every normal matrix $A\in
M_n(\A)$ can be diagonalized, i.e., there are $U\in U_n(\A)$ and $d_1,\cdots,d_n\in\A$ such that $UAU^*=\diag(d_1,\cdots,d_n).$
He also showed that if $\A$ is only assumed to be a $C^*$--algebra, the result may fail. At same time, Grove and Pederson in
\cite{GP} considered the problem of diagonlization of normal matrices over $C(M)$, where $M$ is
a compact Hausdorff space. They characterized the $M$ that allows diagonalization of every
normal element in $M_n(C(M))$. Of course, these conditions in \cite{GP} are very complicated and
hard to be verified.

Recently, because of the study of the classification of $\mathrm{AI}$--algebras and $\mathrm{AT}$--algebras, the approximate
diagonalization of the self--adjoint matrix over $C([0,1])$ and some kind of unitary
matrix over $C(\mathbf{S^1})$ are given respectively (cf. \cite[Example 3.1.6]{Ro}). For more
general case, Choi and Elliott showed that if the dimension of the compact Hausdorff space $M$
is no more than two,  then every self--adjoint element in $M_n(C(M))$ can been approximated by a
self--adjoint element with $n$ distinct eigenvalues in \cite{CE}. Using this result, Thomsen
proved that if $\dim M\le 2$ and $\mathrm H^2(M,\mathbb Z)\cong 0$, then two self--adjoint
elements $a,\,b\in M_n(C(M))$ are approximately unitarily equivalent iff $a(x)$ and $b(x)$ have
same eigenvalues, $\forall\,x\in M$ (cf. \cite[Theorem 1.2, Corollary 1.3]{Th}).

In this paper, we first show that if every self--adjoint matrix over $C(M)$ is approximately
diaonalizable, then $\dim M\le 2$ and $\mathrm H^2(M,\mathbb Z)\cong 0$ and then we give a
constructed proof of Choi and Elliotts' result mentioned above.

\section{Preliminaries}
Throughout the paper, $\A$ is a $C^*$--algebra with unit $1$ and $M$ is a compact Haudorff space.

We view $\A^n$ as the set of all $n\times 1$ matrices over $\A$. Set
\begin{align*}
S_n(\A)&=\{(a_1,\cdots,a_n)^T\in\A^n\vert\,\sum\limits^n_{i=1}a_i^*a_i=1\}, \\
\Lg_n(\A)&=\{(a_1,\cdots,a_n)^T\in\A^n\vert\,\sum\limits^n_{i=1}b_ia_i=1,\
\text{for some}\ b_1,\cdots,b_n\in\A\}.
\end{align*}
According to \cite{R1} and \cite{R2}, the topological stable rank, the cnnected stable
rank and the general stable rank of $\A$ are defined respectively as follows:
\begin{align*}
\tsr(\A)=&\min\{\,n\in\mathbb N\vert\, \A^m\ \text{ is dense in}\ \Lg_m(\A),
                     \forall m\ge n\,\}\\
\csr(\A)=&\min\{\,n\in\mathbb N\vert\, U^0_m(\A)\ \text{ acts transitively on}\
        S_m(\A),\forall m\ge n\,\}\\
\gsr(\A)=&\min\{\,n\in\mathbb N\vert\, U_m(\A)\ \text{ acts transitively on}\
                  S_m(\A),\forall m\ge n\,\}.
\end{align*}
If no such integer exists, we set $\tsr(\A)=\infty$, $\csr(\A)=\infty$ and
$\gsr(\A)=\infty$ respectively. Those stable ranks of $C^*$--algebras are very
usful tools in computing $K$--groups of $C^*$--algebras (cf. \cite{R2}, \cite{X1}, \cite{X2} and \cite{X3} etc.). From \cite{R1} and \cite{N}, we have

\begin{Lem}\label{La}
Let $\A$ be a unital $C^*$--algebra and $M$ be a compact Hausdorff space. Then
\begin{enumerate}
\item [$(1)$] $\gsr(\A)\le\csr(\A)\le\tsr(\A)+1;$
\item [$(2)$] $\tsr(C(M))=\Big[\dfrac{\dim M}{2}\Big]+1,\
      \csr(C(M))\le\Big[\dfrac{\dim M+1}{2}\Big]+1;$
\end{enumerate}
\end{Lem}

Using topological stable rank and general stable rank, we can deduce a key lemma of this
paper as follows:
\begin{Lem}\label{Lb}
Suppose that $\tsr(\A)\le 2$ and $\gsr(\A)\le 2$. Let $A=(a_{ij})_{n\times n}\ (n\ge 3)$ be
self--adjoint element in $\text{M}_n(\A)$. Then for any $ßepsilon>0$, there are $U\in U_n(\A)$
and $b_1,\cdots,b_{n-1}\in\A$ with $b_1,\cdots,b_{n-2}>0$ such that
$$
\Bigg\|U^*AU-\begin{bmatrix} a_{11}& b_1\\ b_1&a_{22}&\ddots\\ \ &\ddots &\ddots&b_{n-1}^*\\
 \ &\ &b_{n-1}&a_{nn}\end{bmatrix}\Bigg\|<(n-1)^{3/2}\epsilon.
$$
\end{Lem}
\begin{proof}Since $\tsr(\A)\le 2$, there exists $(a_{21}^{(2)},\cdots,a_{n1}^{(2)})^T\in
\Lg_{n-1}(\A)$ such that $\|a_{i1}-a_{i1}^{(2)}\|<\epsilon$, $i=2,\cdots,n$. Since $b_1=
\Big[\sum\limits^n_{i=2}\big(a_{i1}^{(2)}\big)^*a_{i1}^{(2)}\Big]^{1/2}$ is invertible,
$(a_{21}^{(2)}b_1^{-1},\cdots,a_{n1}^{(2)}b_1^{-1})^T\in S_{n-1}(\A)$. Thus, from $\gsr(\A)
\le 2$, we get $U_1\in U_{n-1}(\A)$ such that
$$
U_1(a_{21}^{(2)}b_1^{-1},\cdots,a_{n1}^{(2)}b_1^{-1})^T=(1,0,\cdots,0)^T\ \text{or}\
U_1(a_{21}^{(2)},\cdots,a_{n1}^{(2)})^T=(b_1,0,\cdots,0)^T.
$$
Set $U^{(1)}=\diag(1,U_1)$. Then using $a_{ij}=a_{ji}^*$, $i,\,j=1,\cdots$, we have
$$
U^{(1)}\begin{bmatrix}
  a_{11}&\big(a_{i1}^{(2)}\big)^*&\cdots&\big(a_{n1}^{(2)}\big)^*\\
  a_{21}^{(2)}&a_{22}&\cdots& a_{2n}\\
  \vdots&\vdots&\ddots&\vdots\\
  a_{i1}^{(2)}&a_{n2}&\cdots& a_{nn}\end{bmatrix}
  \Big(U^{(1)}\Big)^*=\begin{bmatrix}a_{11}&b_{(1)}^T\\ b_{(1)}&A_1\end{bmatrix},
$$
where $b_{(1)}=(b_1,0,\cdots,0)^T\in\A^{n-1}$, $A_1=A^*_1\in\text{M}_{n-1}{\A}$. Simple
computation shows that
$$
\Big\|U^{(1)}A (U^{(1)})^*-\begin{bmatrix}a_{11}&b_{(1)}^T\\ b_{(1)}&A_1\end{bmatrix}\Big\|<
\sqrt{n-1}\,\epsilon.
$$
Using the same way as above to $A_1$, after $n-2$ times, we can find a $U\in U_n(\A)$ and
invertible positive elements $b_1,\cdots,b_{n-2}\in\A$ and an element $b_{n-1}\in\A$
such that
$$
\Bigg\|U^*AU-\begin{bmatrix} a_{11}& b_1\\ b_1&a_{22}&\ddots\\ \ &\ddots &\ddots&b_{n-1}^*\\
 \ &\ &b_{n-1}&a_{nn}\end{bmatrix}\Bigg\|<(\sqrt{2}+\cdots+\sqrt{n-1})
<(n-1)^{3/2}\epsilon.
$$
\end{proof}

The following lemma concers the extension of continuous map.
\begin{Lem}\label{Lc}
Let $M_0$ be a closed subset of $M$ and $f_0\colon M_0\rightarrow U( \A)$ be a continuous map.
Then there are an open subset $O$ in $M$ containing $M_0$ and a continuous map $f\colon O
\rightarrow U(\A)$ such that $f\big\vert_{M_0}=f_0$.
\end{Lem}
\begin{proof}By \cite[P360]{D}, there is a continuous map $g\colon M\rightarrow\A$ such that
$g\big\vert_{M_0}=f_0$. Thus for any $\epsilon\in(0,\frac{1}{2})$ and any $x_0\in M_0$, there
is an open subset $V(x_0)$ in $M$ containing $x_0$ such that $\|g(x)-g(x_0)\|<\epsilon$
whenever $x\in V(x_0)$. Since $M_0=\dcup_{x_0\in M_0}(V(x_0)\cap M_0)$ and $M_0$ is compact,
it follows that there are $x_1,\cdots,x_n\in M_0$ such that $M_0=\dcup_{i=1}^n(V(x_i)\cap M_0)$.
Set $O= \dcup_{i=1}^nV(x_i)$. Then $O$ is open and contains $M_0$. Furthermore, for any $x\in M$, there is $x_i$ such that
\begin{equation}\label{eaa}
\|g(x)-g(x_i)\|=\|g(x)-f_0(x_i)\|<\epsilon<\dfrac{1}{\,2\,}.
\end{equation}
(\ref{eaa}) implies that $(f_0(x_i))^*g(x)$ is invertible in $\A$ and so is the $g(x)$ and
moreover,
$$
\|(g(x))^{-1}\|=\|[f_0(x_i))^*g(x)]^{-1}\|<2,\ \forall\,x\in O.
$$
Now set $f(x)=g(x)[(g(x))^*g(x)]^{-1/2}$, $x\in O$. Then $f\colon O
\rightarrow U(\A)$ is continuous and $f\big\vert_{M_0}=f_0$ (for $g\big\vert_{M_0}=f_0$).
\end{proof}

The following results are well--known in Matrix Thoery, which come from \cite{W}.
\begin{Lem}\label{Ld}
Let $A$ be a self--adjoint matrix in $\text{M}_n(\mathbb C)$.
\begin{enumerate}
\item[$(1)$]If $A$ is tri--diagonal such that the elements in subdiagonal line are nonzero,
then $A$ has $n$ distinct eigenvalues$;$
\item[$(2)$] Let $\lambda_1\ge\cdots\ge\lambda_n$ be the eigenvalues of $A$, ordered
non--increasingly and each eigenvalue repeated according to its multiplicity. For every
nonzero subspace $V\subset\mathbb C^n$, set
$\lambda_A(V)=\min\{(Ax,x)\vert\,x\in V,\ \|x\|=1\}.$ Then
$$\lambda_j=\max\{\lambda_A(V)\vert\,V\subset\mathbb C^n,\,\dim V=j\},\ j=1,\cdots,n.$$
\end{enumerate}
\end{Lem}

Now applying Lemma \ref{Ld} to self--adjoint matrices in $\text{M}_n(C(M))$, we have
\begin{Cor}\label{Ca}
Let $A,\,B$ be two self--adjoint matrices in $\text{M}_n(C(M))$. For each $x\in M$,
let $\lambda_1(x)\ge\cdots\ge\lambda_n(x)$, $\mu_1(x)\ge\cdots\ge\mu_n(x)$ be the
eigenvalues of $A(x)$ and $B(x)$, ordered non--increasingly and counted with its multiplicity,
respectively. Then, for every $x,\,y\in M$ and $j=1,\cdots,n$,
\begin{enumerate}
\item[$(1)$] $\vert\lambda_j(x)-\lambda_j(y)\vert\le\|A(x)-A(y)\|;$
\item[$(2)$] $\vert\lambda_j(x)-\mu_j(x)\vert\le\|A(x)-B(x)\|.$
\end{enumerate}
\end{Cor}
\begin{proof} We only prove (2). The proof of (1) is similar.

Let $V$ be any nonzero subspace of $\mathbb C^n$. Then for each $x\in M$ and $\xi\in V$,
\begin{align*}
(A(x)\xi,\xi)&\le\|A(x)-B(x)\|\|\xi\|^2+(B(x)\xi,\xi),\\
(A(x)\xi,\xi)&\ge -\|A(x)-B(x)\|\|\xi\|^2+(B(x)\xi,\xi)
\end{align*}
Thus,
$$
\lambda_{A(x)}(V)\le\|A(x)-B(x)\|+ \lambda_{B(x)}(V), \quad
\lambda_{A(x)}(V)\ge -\|A(x)-B(x)\|+ \lambda_{B(x)}(V)
$$
$\forall\,x\in M$ and consequently, by Lemma \ref{Ld},
$$
\lambda_j(x)\le\|A(x)-B(x)\|+ \lambda_j(x),\quad \mu_j(x)\le \|A(x)-B(x)\|+\lambda_j(x),
$$
i.e., $ \vert\lambda_j(x)-\mu_j(x)\vert\le\|A(x)-B(x)\|$, $\forall\,x\in M$ and $j=1,\cdots,n$.
\end{proof}

Let $\mathrm{Det}$ (resp. $\mathrm{Tr}$) denote the determinant (resp. trace) on $\text{M}_n
(\mathbb C)$. Define functions $\det,\ \tr\colon \text{M}_n(C(M))\rightarrow C(M)$ by
$$
\det(A)(x)=\mathrm{Det}(A(x)),\quad \tr(A)(x)=\mathrm{Tr}(A(x)),\ \forall\,A\in
\text{M}_n(C(M)),\ x\in M
$$
respectively. By means of some theory in Linear Algebra, we have
\begin{Lem}\label{Le}
Let $A,\,B\in \text{M}_n(C(M))$. Then
\begin{enumerate}
\item[$(1)$] $\det(AB)=\det(A)\,\det(B)$, $\tr(A+B)=\tr(A)+\tr(B);$
\item[$(2)$] $\tr(AB)=\tr(BA)$, $\|\tr(A)-\tr(B)\|\le n\,\|A-B\|;$
\item[$(3)$] $\|\det(A)-\det(B)\|\le n!\Big(\sum\limits^{n-1}_{k=0}\|A\|^{k}\|B\|^{n-k-1}\Big)
\|A-B\|;$
\item[$(4)$] $A$ is invertible in $\text{M}_n(C(M))$ iff $\mathrm{Det}(A(x))\not=0$, $\forall\,
x\in M$.
\end{enumerate}
\end{Lem}

\section{A necessary condition}
\setcounter{equation}{0}

\begin{Def}
An element $A\in\text{M}_n(\A)$ $(n\ge 2)$ is said to be approximately
diagonalizable, if for any $\epsilon>0$, there are $U\in U_n(\A)$ and $a_1,\cdots,a_n\in\A$ such that
\begin{equation}\label{ebb}
\|UAU^*-\diag(a_1,\cdots,a_n)\|<\epsilon.
\end{equation}

$\A$ is called to be approximately diagonal $\mathbf{(AD)}$, if for any $n\ge 2$, every self--adjoint element in $\text{M}_n(\A)$ $(n\ge 2)$ can be approximate diagonalization.
\end{Def}

Clearly, if $A$ is self--adjoint (or unitary), then $a_1,\cdots,a_n$ in
(\ref{ebb}) can be chosen as self--adjoint (or unitary).
\begin{Lem}\label{L2a}
Let $P$ be a approximately diagonalizable projection in $\text{M}_n(\A)$
$(n\ge 2)$. Then there are $U\in U_n(\A)$ and projections $p_1,\cdots,p_n\in\A$ such that $UPU^*=\diag(p_1,\cdots,p_n)$.
\end{Lem}
\begin{proof}By assumption, there are $W\in U_n(\A)$ and self--adjoint elements
$a_1,\cdots,$ $a_n\in A$ such that
$$
\|WPW^*-\diag(a_1,\cdots,a_n)\|<\dfrac{1}{\,2\,}.
$$
Put $A=\diag(a_1,\cdots,a_n)$. Then by \cite[Lemma 2.5.4]{L4}, there exists a
projection $Q$ in the $C^*$--algebra generated by $A$ such that
\begin{equation}\label{ecc}
\|WPW^*-Q\|<2\,\|WPW^*-A\|<1.
\end{equation}
Since $A$ is diagonal, $Q$ has the form $Q=\diag(p_1,\cdots,p_n)$, where
$p_1,\cdots,p_n\in\A$ are projections. Finally, Using \cite[Lemma 2.5.1]{L4}
to (\ref{ecc}), we can find $W_0\in U_n^0(\A)$ such that $W_0WPW^*W_0^*=
\diag(p_1,\cdots,p_n)$. Put $U=W_0W$. Then we get the assertion.
\end{proof}

Recall from \cite{HW} that a compact Hausdorff space $M$ is of $\dim M\le n$
iff for each closed subset $A$ of $M$, any continuous map $f\colon A
\rightarrow\mathbf{S^n}$ can be extended to $M$ and if $\dim M<\infty$, then
$\dim M\le n$ iff $i^*\colon\HO^n(M,\mathbb Z)\rightarrow\HO^n(A,\mathbb Z)$
is epimorphic for any closed subset $A$ of $M$, where $\HO^n(M,\mathbb Z)$ is
the n'th C\v ech Cohomology of $M$ and $i^*$ is the induced homomorphism of
the inclusion $i\colon A\rightarrow M$ on $\HO^n(M,\mathbb Z)$.
\begin{Lem}\label{L2b}
Suppose that $C(M)$ is $\mathrm{(AD)}$. Then $\dim M\le 3$.
\end{Lem}
\begin{proof}If $\dim M>3$, then exist a closed subset $M_0$ of $M$ and a
continuous map $f_0\colon M_0\rightarrow\mathbf{S^3}$ which can not be extended
to $M$. Write $f_0(x)=(f_1(x),f_2(x))^T$, $x\in M_0$, where $f_1,f_2\colon M_0
\rightarrow\mathbb C$ are continuous and $|f_1(x)|^2+ |f_2(x)|^2=1$, $\forall\,
x\in M_0$. Put
$$
U_0(x)=
\begin{bmatrix}f_1(x)&-\overline{f_2(x)}\\ f_2(x)&\overline{f_1(x)}\end{bmatrix},
\quad x\in M_0.
$$
Then $U_0\in U_2(C(M_0))$ and it follows from Lemma \ref{Lc} that there are
an open subset $O$ in $M$ containing $M_0$ and a continuous map $V\colon O
\rightarrow U_2(\mathbb C)$ such that $V\big\vert_{M_0}=U_0$.
Pick a continuous function $h\colon M\rightarrow [0,1]$ such that $h(x)=1$, if
$x\in M_0$ and $h(x)=0$ when $x\in O\backslash M_0$. Now define a normal
element $N\in\text{M}_2(C(M))$ by
$$
N(x)=\begin{cases} h(x)V(x)\qquad& x\in O\\ 0\qquad& x\in M\backslash O \end{cases}
$$
and put $A_1=\dfrac{1}{\,2\,}(N+N^*)$, $A_2=\dfrac{1}{2i}(N-N^*)$. By hypothesis, there are $W\in U_2(C(M))$ and continuous real valued functions $\lambda_1,\,\lambda_2$ on $M$ such that
\begin{equation}\label{edd}
\|WA_2W^*-\diag(\lambda_1,\lambda_2)\|<\dfrac{1}{\,6\,}.
\end{equation}
Note that $A_1=\diag(h\mbox{Re}(f_1),h\mbox{Re}(f_1))$. So if we set
$\mu_j=h\mbox{Re}(f_1)+i\lambda_j$, $j=1,2$, then
\begin{equation}\label{eee}
\|WAW^*-\diag(\mu_1,\mu_2)\|<\dfrac{1}{\,6\,},
\end{equation}
by (\ref{edd}). Applying Lemma \ref{Le} to (\ref{eee}), we get that
\begin{align*}
\|h^2-\mu_1\mu_2\|\le& 2\,(\|WNW^*\|+\|\diag(\mu_1,\mu_2)\|)\\
                  \le& 2(1+1+\dfrac{1}{\,6\,})\,\dfrac{1}{\,6\,}<1
\end{align*}
Set
$$
X=\begin{bmatrix}\mu_1&-(1-h^2)^{1/2}\\ (1-h^2)^{1/2}&\mu_2\end{bmatrix}\in
\text{M}_2(C(M)).
$$
Since $\|1-\det(X)\|=\|h^2-\mu_1\mu_2\|<1,$ $X$ is invertible by Lemma \ref{Le}.
Set $W_1=X(X^*X)^{-1/2}\in U_2(C(M))$. Noting that $h(x)=1$, when $x\in M_0$, we have from
(\ref{eee}),
$$
\|W(x)V(x)(W(x))^*-X(x)\|<\dfrac{1}{\,6\,},\quad x\in M_0.
$$
Simple computation shows that
$$
\|1_2-(X(x))^*X(x)\|<\dfrac{13}{36},\quad \|1_2-((X(x))^*X(x))^{-1/2}\|<
\dfrac{1}{\,2\,},\ \forall\,x\in M_0.
$$
Thus,
$$
\|W(x)V(x)(W(x))^*-W_1(x)\|<\dfrac{1}{\,6\,}+\dfrac{1}{\,2\,}(1+\dfrac{1}{\,6\,})<1,
\quad x\in M_0.
$$
This means that there is a self--adjoint element $B_0$ in $\text{M}_2(C(M_0))$ such that
$$
W(x)V(x)(W(x))^*=\text{exp}(i\,B_0(x))W_1(x),\quad\forall\,x\in M_0.
$$
Choose self--adjoint element $B$ in $\text{M}_2(C(M))$ such that $B\big\vert_{M_0}=B_0$
and put
$$
U=W^* \text{exp}(i\,B)W_1W=\begin{bmatrix}u_{11}&u_{12}\\ u_{21}&u_{22}\end{bmatrix}
\in U_2(C(M)),
$$
$f=(u_{11},u_{21})^T$. Then $f\colon M\rightarrow\mathbf{S^3}$ is continuous with
$f\big\vert_{M_0}=f_0$, a contradicition.
\end{proof}

Let $\text{M}_\infty(C(M))=\dcup^\infty_{n=1}\text{M}_n(C(M))$ under the
inclusion
$$
i_n\colon\text{M}_n(C(M))\rightarrow\text{M}_{n+1}(C(M))\ \mbox{given by}\
i_n(a)=\diag(a,0).
$$
It is well--known that there is an one--to--one and onto correspodence between $n$--dimensional complex vector bundles over $M$ and projections in $\text{M}_\infty(C(M))$ with rank $n$ (i.e., $p$ is a projection
in $\text{M}_m(C(M))$ with m large enough such that $\tr(p)=n$). Let
$\mathrm{V}^1_{\mathbb C}(M)$ denote the isomorphic class of all 1--dimensional complex vector bundles over $M$. Then from \cite{K}, $\mathrm{V}^1_{\mathbb C}(M)\cong\HO^2(M,\mathbb Z)$. Thus $\HO^2(M,\mathbb Z)\cong 0$ iff all 1--dimensional complex vector bundles over $M$ is trivial iff ever projection
with rank one in $\text{M}_\infty(C(M))$ is equivalent to the form $\diag(1,0_s)\in\text{M}_{s+1}(C(M))$ for sufficiently large $s$.
\begin{thm}\label{Tha}
Let $M$ be a compact Hausdorff space such that $C(M)$ is $\mathrm{(AD)}$. Then $\dim M\le 2$ and $\HO^2(M,\mathbb Z)\cong 0$.
\end{thm}
\begin{proof}Let $A$ be any closed subspace of $M$ and let $a$ be any self--adjoint element in $\text{M}_n(C(A))$ for $n\ge 2$. Then we can find a
self--adjoint element $\hat a\in \text{M}_n(C(M))$ such that $\hat a\big\vert_A
=a$. This means that $C(A)$ is $\mathrm{(AD)}$ when $C(M)$ is $\mathrm{(AD)}$.

Now let $p$ be a projection in $\text{M}_n(C(A))$ with $\rank(p(x))=1,\ \forall\,x\in A$. Then by Lemma \ref{L2a}, there are projections $p_1,\cdots,
p_n\in C(A)$ and $W\in\text{M}_n(C(A))$ such that
\begin{equation}\label{eff}
WpW^*=\diag(p_1,\cdots,p_n)
\end{equation}
Thus,
$$
p_1(x)+\cdots+p_n(x)=\mathrm{Tr}(\diag(p_1(x),\cdots,p_n(x))
                     =\mathrm{Tr}(p(x))=1,
$$
$\forall\,x\in A$ and hence $p_ip_j=0,\ i\not=j,\ i,j=1,\cdots,n$. Set
$$
S=\begin{bmatrix}p_1&\cdots&p_n\\ 0&\cdots&0\\ \vdots&\ddots&\vdots\\
                  0&\cdots&0\end{bmatrix} W.
$$
Then $S^*S=p,\ SS^*=\diag(1,0_{n-1})$ by (\ref{eff}). This shows that
$\HO^2(A,\mathbb Z)\cong 0$.

Since $\dim M\le 3$ by Lemma \ref{L2b} and $i^*\colon\HO^2(M,\mathbb Z)
\rightarrow\HO^2(A,\mathbb Z)$ is surjective, it follows that $\dim M\le 2$.
\end{proof}

\section{A sufficient condition}
\setcounter{equation}{0}

In this section, we will prove following theorm:
\begin{thm}\label{Thb}
Let $M$ be compact Hausdorff space with $\dim M\le 2$ and $\HO^2(M,\mathbb Z)$
$\cong 0$. Then $C(M)$ is $\mathrm{(AD)}$. Precisely, let $A$ be a self--adjoint
element in $\text{M}_n(C(M))$ $(n\ge 2)$ and let $\lambda_1(x)\ge\cdots\ge
\lambda_n(x)$ be the eigenvalues of $A(x)$, ordered non--increasingly and counted with their multiplicity for each $x\in M$. Then for any $\epsilon>0$,
there is $U\in U_n(C(M))$ such that
$$
\|U^*(x)A(x)U(x)-\diag(\lambda_1(x),\cdots,\lambda_n(x))\|<\epsilon,\
\forall\,x\in M.
$$
\end{thm}

To prove this theorm, we need two lemmas.
\begin{Lem}{\rm\cite[Proposition 1.1]{BP}}\label{L3a}
Let $M$ be a compact Hausdorff space with $\dim M\le n$. Then for any real--valued continuous fuctions $f_1,\cdots,f_{n+1}$ on $M$, there are
real--valued continuous functions $g_1,\cdots,g_{n+1}$ on $M$ such that
$\sum\limits^{n+1}_{i=1}g_i^2$ is invertible in $C(M)$ and $\|f_i-g_i\|<\epsilon,\ i=1,\cdots,n+1$.
\end{Lem}

Let $H_1,H_2$ be two complex Hilbert spaces and let $B(H_1,H_2)$ denote the set of all bounded linear operators from $H_1$ to $H_2$. For $T\in
B(H_1,H_2)$, we denote by $\Ker T$ (resp. $\ran(T)$) the null space (resp. range) of $T$.
\begin{Lem}\label{L3b}
Let $T(x)$ be a continuous map from $M$ to $B(H_1,H_2)$ such that $\dim\Ker  T(x)=n$ and $\ran(T(x))$ is closed in $H_2$, $\forall\,x\in M$. Set
$$
E(T)=\{(x,\xi)\in M\times H_1\vert\, T(x)\xi=0\},\ \pi(x,\xi)=x,\
\forall\,(x,\xi)\in E_T.
$$
Then $(E(T),\pi,M)$ is an $n$--dimensional complex vector bundle over $M$.
\end{Lem}
\begin{proof} Let $x_0\in M$ be an arbitrary point. Let $P$ (resp. $Q$)
be the projection of $H_1$ (resp. $H_2$) onto $\Ker T(x_0)$ (resp.
$\ran(T(x_0)$). Then there is $G\in B(H_2,H_1)$ such that $GT(x_0)=I_{H_1}-P$, $T(x_0)G=Q$
($G$ is called the generalized inverse of $T(x_0)$, denoted by $T(x_0)^+$).

Since $T(x)$ is continuous at $x_0$, we can find a closed neighbourhood $U(x_0)$ of $x_0$ in $M$ such that $\|T(x)-T(x_0)\|<\dfrac{1}{\|A\|}$ whenever $x\in U(x_0)$. Thus, $\phi(x)=(I_{H_1}+G(T(x)-T(x_0)))^{-1}$ is
a continuous map from $U(x_0)$ to the the group of invertible operators in
$B(H_1)$ and furthermore, $\Ker T(x)=\phi(x)\Ker T(x_0)$ by \cite[Proposition 3.1]{CX}.
Now let $\{e_1,\cdots,e_n\}$ be a basis for
$\Ker T(x_0)$ and put $e_j(x)=\phi(x)e_j$, $j=1,\cdots,n$, $x\in U(x_0)$.
Then $\{e_1(x),\cdots,e_n(x)\}$ forms a continuous basis for $\Ker T(x)$,
$x\in U(x_0)$. This shows that $(E(T),\pi,M)$ is a vector bundle of
dimension $n$ (cf.\cite{H}).
\end{proof}

Proof of Theorem \ref{Thb}. We firt assume that $A=\begin{bmatrix}
a_{11}&a_{12}^*\\ a_{21}&a_{22}\end{bmatrix}$ is self--adjoint in
$\text{M}_2(C(M))$. Write $a_{12}=a^{(1)}_{12}+i\,a_{12}^{(2)}$. Since $\dim
M\le 2$, it follows from Lemma \ref{L3a} that for any $\epsilon>0$, there are
continuous functions $b_1,b_2,b_3\colon M\rightarrow\mathbb R$ such that
$$
\|a_{11}-a_{22}-b_1\|<\epsilon,\ \|a^{(1)}_{12}-b_2\|<\epsilon,\
\|a^{(2)}_{12}-b_3\|<\epsilon
$$
and $b_1^2+b_2^2+b_3^2$ is invertible. Set $b=b_2+ib_3$ and
$B=\begin{bmatrix}a_{11}&b^*\\ b&a_{11}-b_1\end{bmatrix}$. Then
\begin{equation}\label{egg}
\|A-B\|=\Big\|\begin{bmatrix}0&(b-a_{12})^*\\ b-a_{12}& a_{22}-a_{11}-b_1
\end{bmatrix}\Big\|<2\,\epsilon
\end{equation}
and  for each $x\in M$, two eigenvalues of $B(x)$
\begin{align*}
\mu_1(x)&=\dfrac{1}{\,2\,}\bigg[2a_{11}(x)-b_1(x)+\sqrt{b^2_1(x)+4\vert b(x)\vert^2}\bigg],\\
\mu_2(x)&=\dfrac{1}{\,2\,}\bigg[2a_{11}(x)-b_1(x)-\sqrt{b^2_1(x)+4\vert b(x)\vert^2}\bigg]
\end{align*}
are not equal for each $x\in M$. Thus, $\Ker(B(x)-\mu_j(x))=1$, $\forall\,x\in M$, $j=1,2$
and hence by Lemma \ref{L3b} $(E(B-\mu_j),\pi,M)$ is an one--dimensional
complex vector bundle which is also trivial for $\HO^2(M,\mathbb Z)\cong 0,\ j=1,2$. This implies that there are continuous maps $\xi_1,\xi_2\colon M
\rightarrow\mathbb C^2$ with $\|\xi_j\|=1$ such that $B(x)\xi_j(x)=\mu_j(x)
\xi_j(x)$, $\forall\,x\in M$ and $j=1,2$. Moreover, from $\mu_1(x)>\mu_2(x)$,
we have $(\xi_1(x),\xi_2(x))=0$, $\forall\,x\in M$. Put $U(x)=(\xi_1(x),\xi_2(x))$, $x\in M$. Then $U\in U_2(C(M))$ and
\begin{equation}\label{ehh}
U^*(x)B(x)U(x)=\begin{bmatrix}\mu_1(x)\\ \ &\mu_2(x)\end{bmatrix}.
\end{equation}
Let $\lambda_1(x),\lambda_2(x)$ be the eigenvalus of $A(x)$, for each $x\in M$.
Then by Corollary \ref{Ca} (2) and (\ref{egg}), $\|\lambda_j-\mu_j\|<2\epsilon,
\ j=1,2$. Finally, combining this with (\ref{egg}) and (\ref{ehh}), we obtain
$\|U^*AU-\diag(\lambda_1,\lambda_2)\|<4\epsilon$.

Now Let $A=(a_{ij})_{n\times n}$ with $a_{ij}\in C(M)$ and $a_{ij}^*=a_{ji}$,
$i,j=1,\cdots,n$, $n\ge 3$. By Lemma \ref{La}, $\tsr(C(M))\le 2$ and $\gsr(C(M))\le 2$ when $\dim M\le 2$. Thus by Lemma \ref{Lb}, for any $\epsilon
\in (0,\frac{1}{2})$, there are $b_1,\cdots,b_{n-1}\in C(M)$ with $b_1,\cdots,
b_{n-2}>0$ and $U_1\in U_n(C(M))$ such that
\begin{equation}\label{eii}
\Bigg\|U_1^*AU_1-\begin{bmatrix} a_{11}& b_1\\ b_1&a_{22}&\ddots\\ \ &\ddots\  &\ddots&b_{n-1}^*\\
 \ &\ &b_{n-1}&a_{nn}\end{bmatrix}\Bigg\|<(n-1)^{3/2}
 \dfrac{\epsilon}{(n-1)^{3/2}}=\epsilon.
\end{equation}
Write $c_j=a_{jj}-a_{nn}$, $j=1,\cdots,n$ and
$$
B_0=\begin{bmatrix} c_1& b_1\\ b_1&c_2&\ddots\\ \ &\ddots&\ddots
   &b^*_{n-1}\\  \ &\ & b_{n-1}&0
\end{bmatrix},\quad
B_k=\begin{bmatrix}c_1& b_1\\ b_1&c_2&\ddots\\ \ &\ddots&\ddots
           &b^*_{k-1}\\  \ &\ & b_{k-1}&c_k,
     \end{bmatrix}
$$
$k=2,\cdots,n-1$. Put $q_k=\det(B_k),\ q_1=c_1,\ q_0=1,\ k=2,\cdots,n-1$. Then
$q_1,\cdots,q_{n-1}$ are real--valued functions in $C(M)$ and
\begin{equation}\label{ejj}
q_k=c_kq_{k-1}-b_{k-1}^2q_{k-2},\ k=2,\cdots,n-1.
\end{equation}
From (\ref{eii}), we can deduce that $q_{k-1}^2+q_{k-2}^2$ is invertible in
$C(M)$, $k=2,\cdots,n-1$. Write $b_{n-1}=b_{n-1}^{(1)}+ib_{n-1}^{(2)}$. Since
$\dim M\le 2$, it follows from Lemma \ref{L3a} that there are real continuous
functions $d_{n-1}^{(1)},d_{n-1}^{(2)},b$ on $M$ such that
$$
\|b_{n-1}^{(j)}-d_{n-1}^{(j)}\|<\epsilon,\quad
\|q_{n-1}-b\|<\dfrac{\min\{\epsilon,m\epsilon,0.5m^2\}}{(\|q_{n-2}\|+\|
q_{n-3}\|)\|(q_{n-2}^2+q_{n-3}^2)^{-1}\|},
$$
$j=1,2$ and $b^2+(d_{n-1}^{(1)})^2+(d_{n-1}^{(2)})^2$ is invertible in $C(M)$,
where $m=\min\limits_{x\in M}b_{n-2}(x)>0$. Note that
\begin{align*}
  b_{n-2}^2(x)-\dfrac{(b(x)-q_{n-1}(x))q_{n-3}(x)}{q_{n-2}^2(x)+q_{n-3}^2(x)}&
  \ge m^2-\|(q_{n-2}^2+q_{n-3}^2)^{-1}(b-q_{n-1})q_{n-3}\|\\
  &\ge m^2-\|(q_{n-2}^2+q_{n-3}^2)^{-1}\|\|(b-q_{n-1})\|\|q_{n-3}\|\\
  &\ge m^2-0.5m^2=0.5m^2>0
\end{align*}
Put $d_{n-1}=d_{n-1}^{(1)}+i\,d_{n-1}^{(2)}$ and
\begin{align*}
c'_{n-1}&=c_{n-1}+\dfrac{(b-q_{n-1})q_{n-2}}{q^2_{n-2}+q^2_{n-3}},\quad
b'_{n-2}=\Big(b^2_{n-2}-\dfrac{(b-q_{n-1})q_{n-3}}{q^2_{n-2}+q^2_{n-3}}
\Big)^{1/2}>0\\
B_{n-1}'&=\begin{bmatrix}B_{n-2}&X\\ X^T&
c'_{n-1}\end{bmatrix},\qquad\qquad\
B=\begin{bmatrix}B'_{n-1}&Y\\ Y^T&0\end{bmatrix},
\end{align*}
where $X=(0,\cdots,0,b'_{n-2})^T\in(C(M))^{n-2},\
Y=(0,\cdots,0,d_{n-1})^T\in(C(M))^{n-1}$.  Then
\begin{align*}
\|c'_{n-1}-c_{n-1}\|<&\epsilon,\ \|b'_{n-2}- b_{n-2}\|<\epsilon,\
\|B'_{n-1}-B_{n-1}\|<3\epsilon\\
\|B-B_0\|\le&\|B'_{n-1}-B_{n-1}\|+2\|d_{n-1}-b_{n-1}\|<6\epsilon\
\text{and}\\
\det(B'_{n-1})=&c'_{n-1}q_{n-2}-(b'_{n-2})^2q_{n-3}=b.
\end{align*}

Now let $\mu(x)$ be an eigenvalue of $B(x)$ for each $x\in M$. Put
$M_1=\{x\in M\vert\,d_{n-1}(x)=0\}$. Let $x_0\in M\backslash M_1$.
Then $d_{n-1}(x_0)=\vert d_{n-1}(x_0)\vert\text{exp}(i\theta)$ for some
$\theta\in\mathbb R$ and
$$
\begin{bmatrix}1_{n-1}\\ \ &\text{exp}(-i\theta)\end{bmatrix}B(x_0)
\begin{bmatrix}1_{n-1}\\ \ &\text{exp}(i\theta)\end{bmatrix}=
\begin{bmatrix}B'_{n-1}(x_0)&Z\\ Z^T& 0\end{bmatrix},
$$
where $Z=(0,\cdots,0,\vert d_{n-1}(x_0)\vert)^T\in(C(M))^{n-1}$.
In this case, $\dim\Ker(B(x_0)-\mu(x_0)I_n)=1$ by Lemma \ref{Ld} (1).
Suppose that $x_0\in M_1$. If $\mu(x_0)\not=0$, then $\mu(x_0)$ must be an eigenvalue of
$B'_{n-1}(x_0)$. Thus, by Lemma \ref{Ld} (1),
$\dim\Ker(B'_{n-1}(x_0)-\mu(x_0)I_{n-1})=1$ and hence
$\dim\Ker(B(x_0)-\mu(x_0)I_n)=1$; If $\mu(x_0)=0$, then from $
b^2(x)+\vert d_{n-1}(x)\vert^2\not=0,\ \forall\,x\in M$ and $\text{Det}
(B'_{n-1}(x_0))=b(x_0),\ d_{n-1}(x_0)=0$, we have
$\text{Det}(B'_{n-1}(x_0))\not=0$, i.e., $0$ is not the eigenvalue of
$B'_{n-1}(x_0)$, so $\dim\Ker(B(x_0)-\mu(x_0)I_n)=1$.

The above shows that $\dim\Ker(B(x)-\mu(x)I_n)=1$ for each $x\in M$.
Let $\mu_1(x)\ge\cdots\ge\mu_n(x)$ be the eigenvalues of $B(x)$, ordered non--increasingly and counted with its multiplicity, for each $x\in M$.
Then $\mu_1(x)>\cdots>\mu_n(x),\ \forall\,x\in M$ and there are continuous maps $\xi_1,\cdots,\xi_n\colon M\rightarrow\mathbb C^n$ with
$\|\xi_i\|=1$ such that $B(x)\xi_i(x)=\mu_i(x)\xi_i(x)$, $\forall\,x\in M$
and $i=1,\cdots,n$ by Lemma \ref{L3b} and the assumption $\HO^2(M,\mathbb Z)\cong 0$. Moreover, $(\xi_i(x),\xi_j(x))=0$, $i\not=j,\ i,j=1,\cdots,n$,
$\forall\,x\in M$. Put $U_2(x)=(\xi_1(x),\cdots,\xi_n(x)),\ t\in M$. Then
$U_2\in U_n(C(M))$ and
$$
U^*_2(x)B(x)U_2(x)=\diag(\mu_1(x),\cdots,\mu_n(x)),\ \forall\,x\in M.
$$
Put $U=U_1U_2$. Then by (\ref{eii}),
\begin{align*}
\|U^*AU-&\diag(\mu_1(x)+a_{nn},\cdots,\mu_n(x)+a_{nn})\|\\
       \le&\|U^*_2(U_1^*AU_1-B_0-a_{nn}1_n)U_2+U^*_2(B_0-B)U_2\\
         &+U_2^*BU_2-\diag(\mu_1(x),\cdots,\mu_n(x))\|\\
       <&\epsilon+6\epsilon=7\epsilon.
\end{align*}
Let $\lambda_1(x)\ge\cdots\ge\lambda_n(x)$ be the eigenvalues of $A(x)$, ordered non--increasingly and counted with its multiplicity, for each $x\in M$. Then we have
$$
\|\lambda_j-\mu_j-a_{nn}\|<7\epsilon,\quad j=1,\cdots,n
$$
by Corollary \ref{Ca} and so that $\|U^*AU-\diag(\lambda_1,\cdots,\lambda_n)\|<
14\epsilon$.
\vspace{2mm}

Now we consider the approximate diagonalization of unitary matrices over
$C(M)$. We have
\begin{Prop}
Let $M$ be compact metric space.
\begin{enumerate}
\item[$(1)$]If $\dim M\le 2$ and $\HO^2(M,\mathbb Z)\cong 0$, then every unitary element in $U_n^0(C(M))$ is approximately diagonalizable$;$
\item[$(2)$] Every unitary matrix over $C(M)$ is approximately diagonalizable
iff $\dim M$ $\le 2$ and $\HO^j(M,\mathbb Z)\cong 0$, $j=1,2$.
\end{enumerate}
\end{Prop}
\begin{proof}(1) Let $U\in U_n^0(C(M))$. By \cite[Theorem 2.1]{Ph}, for any
$\epsilon\in (0,1)$, there is self--adjoint element $A$ in $\text{M}_n(C(M))$
such that $\|U-\text{exp}(i A)\|<\epsilon/2$. Since $A$ can be approximate
diagonalization by Theorem \ref{Thb}, we get that by simple computation,
$U$ is approximately diagonalizable.

(2) Assume that $\dim M\le 2$ and $\HO^j(M,\mathbb Z)\cong 0$, $j=1,2$.
Then we have $\csr(C(M))\le 2$ and for any $U\in U_n(C(M))$ there are
$U_0\in U_n^0(CM))$ and $a\in U(C(M))$ such that $U=U_0\diag(1_{n-1},a)$
by \cite[Lemma 2.1]{X4}. Note that $U(C(M))\slash U_0(C(M))\cong\HO^1(M,\mathbb Z)$
(Arens and Royden's Theorem). So $U_n(C(M))$ is connected $n\ge 2$ when $\HO^1(M,\mathbb Z)\cong 0$. Then the claim follows from (1).

On the other hand, if ever unitary matrix over $C(M)$ is approximately diagonalizable, then for any self--adjoint matrix $A\in\text{M}_n(C(M))$,
$n\ge 2$, we can deduce from following equation
$
V=\dfrac{A}{\|A\|}+i\sqrt{1_n-\Big(\dfrac{A}{\|A\|}\Big)^2}\in U_n(C(M))
$
that $A=\dfrac{1}{\,2\,}\|A\|(V+V^*)$ is approximately diagonalizable and hence
$\dim M\le 2$ and $\HO^2(M,\mathbb Z)\cong 0$ by Theorem \ref{Tha}.

Now we prove $\HO^1(M,\mathbb Z)\cong 0$. Let $f \in U(C(M))$ such that the
equivalece class $[f]$ of $f$ in $U(C(M))\slash U_0(C(M))$ is a generator.
Since the unitary matrix $U=\begin{bmatrix}0&1\\ f&0\end{bmatrix}$ is
approximately diagonalizable, there are $\lambda_1,\lambda_2\in U(C(M))$ and
$W\in U_2(C(M))$ such that
\begin{equation}\label{ekk}
\|WUW^*-\diag(\lambda_1,\lambda_2)\|<\dfrac{1}{\,6\,}.
\end{equation}
By using Lemma \ref{Le} to (\ref{ekk}), we get that
$$
\|\lambda_1+\lambda_2\|<\dfrac{1}{\,3\,},\quad\|-f-\lambda_1\lambda_2\|
<\dfrac{2}{\,3\,}
$$
and hence $\|f-\lambda^2_1\|<1$. Put $\lambda=\vert\lambda_1\vert^{-1}\lambda_1$. Then $[f]=2[\lambda]$ in
$U(C(M))\slash U_0(C(M))$. Noting that $\HO^1(M,\mathbb Z)$ is torsion--free,
we have $[f]=0$.
\end{proof}

\end{document}